\documentclass[a4paper,11pt]{article}
\usepackage[affil-it]{authblk}
\usepackage{amsmath}

\usepackage{caption}
\usepackage{amssymb}

\usepackage{amsthm,aliascnt,hyperref}
\numberwithin{equation}{section}
\newcommand{\mynewtheorem}[2]{
  \newaliascnt{#1}{dummy}
  \newtheorem{#1}[#1]{#2}
  \aliascntresetthe{#1}
  \expandafter\def\csname #1autorefname\endcsname{#2}
}
 
\theoremstyle{plain}
  \mynewtheorem{theorem}{Theorem}
  \mynewtheorem{proposition}{Proposition}
  \mynewtheorem{corollary}{Corollary}
  \mynewtheorem{lemma}{Lemma}
\theoremstyle{definition}
  \mynewtheorem{definition}{Definition}
  \mynewtheorem{problem}{Problem}
\theoremstyle{remark}
  \mynewtheorem{remark}{Remark}
\newenvironment{thmbis}
  {\addtocounter{theorem}{-1}%
   \begin{theorem}}
  {\end{theorem}}
\usepackage{amsfonts}
\usepackage{array}
\usepackage{verbatim}
\usepackage{enumerate}

\usepackage{twoopt}

\usepackage{color}
\newcommand{\norm}[1]{\left\|{#1}\right\|}

\newcommand{\ip}[2]{\left<#1,#2\right>}  

\newcommand{\tos}{\xrightarrow{s}}  

\newcommand{\tosr}{\xrightarrow{s.r.}}
\newcommand{\tonr}{\xrightarrow{n.r.}}

\newcommand{\dom}[1]{\spac{D}({#1})} 
\newcommand{\bounded}[1]{\spac{B}\left({#1}\right)}  

\newcommand{\+}[1]{\ensuremath{\boldsymbol{#1}}} 

\DeclareMathOperator{\supp}{supp}

\newcommand{\oper}[1]{\mathcal{#1}} 
\newcommand{\spac}[1]{\mathfrak{#1}}
\newcommand{\sform}[1]{\mathfrak{\lowercase{#1}}}

\newcommandtwoopt{\Me}[2][\lambda][\varepsilon]{\ensuremath{\+{\oper{M}}^{#1}_{#2}}} 

\title{Approximations of strongly continuous families of unbounded self-adjoint operators}
\author{Jonathan Ben-Artzi \thanks{Email: \texttt{J.Ben-Artzi@imperial.ac.uk}}} 
\affil{Department of  Mathematics\\Imperial College London}
\author{Thomas Holding \thanks{Email: \texttt{T.J.Holding@maths.cam.ac.uk}}}
\affil{Cambridge Centre for Analysis\\University of Cambridge}

\begin{document}
\renewcommand{\thefootnote}{\fnsymbol{footnote}} 
\footnotetext{\emph{2010 Subject class:} Primary 35P20}     
\renewcommand{\thefootnote}{\arabic{footnote}} \maketitle

\begin{abstract}
The problem of approximating the discrete spectra of families of self-adjoint operators that are merely strongly continuous is addressed. It is well-known that the spectrum need not vary continuously (as a set) under strong perturbations. However, it is shown that under an additional compactness assumption the spectrum does vary continuously, and a family of symmetric finite-dimensional approximations is constructed. An important feature of these approximations is that they are valid for the entire family uniformly. An application of this result to the study of plasma instabilities is illustrated.
\end{abstract}

\tableofcontents

\section{Introduction}
\subsection{Overview}
We present a method for obtaining finite-dimensional approximations of the discrete spectrum of families of self-adjoint operators. We are interested in operators that decompose into a system of two coupled Schr\"odinger operators with opposite signs (see \eqref{eq:non-positive} below). However our results are applicable to ``standard'' Schr\"odinger operators, and in fact we prove our main result, \autoref{thm:main}, for Schr\"odinger operators first, see \autoref{thm:main2}. We are interested in the following problem:
\begin{problem}\label{main problem}
Consider the family of self-adjoint unbounded operators
	\begin{equation}\label{eq:non-positive}
	\oper{M}^\lambda=\oper{A}+\oper{K}^\lambda=
	\begin{bmatrix}
	-\Delta+1&0\\
	0&\Delta-1
	\end{bmatrix}
	+
	\begin{bmatrix}
	\oper{K}^\lambda_{++}&\oper{K}^\lambda_{+-}\\
	\oper{K}^\lambda_{-+}&\oper{K}^\lambda_{--}
	\end{bmatrix},\quad \lambda\in [0,1]
	\end{equation}
acting in an appropriate subspace of $L^2(\mathbb{R}^d)\oplus L^2(\mathbb{R}^d)$, where  $\{\mathcal{K}^\lambda\}_{\lambda\in [0,1]}$ is a bounded, symmetric and strongly continuous family.  \emph{Is it possible to construct explicit finite-dimensional self-adjoint approximations of $\mathcal{M}^\lambda$ whose spectrum in compact subsets of $(-1,1)$ converges to that of $\mathcal{M}^\lambda$ uniformly in $\lambda$?}
\end{problem}

This problem is motivated by Maxwell's equations, which in the Lorenz gauge may be written as the following elliptic system for the electromagnetic potentials $\phi$ and $\mathbf{A}$ (after taking a Laplace transform in time):
	\begin{equation}\label{eq:maxwell-intro}
	\left\{\begin{aligned}
	(-\mathbf{\Delta}+\lambda^2)\mathbf{A}+\mathbf{j}=\mathbf{0}\\
	(\Delta-\lambda^2)\phi+\rho=0
	\end{aligned}\right.
	\end{equation}
where $\rho$ and $\mathbf{j}$ are the charge and current densities, respectively. The specific problem we have in mind, treated separately in  \cite{Ben-Artzi2013a}, is that of instabilities of the relativistic Vlasov-Maxwell system describing the evolution of collisionless plasmas {and} it is outlined in \autoref{sec:application} below. The Vlasov equation provides the coupling of the two equations   in \eqref{eq:maxwell-intro}, making the system self-adjoint (see, for instance, the expressions \eqref{eq:gauss} and \eqref{eq:maxwell2}).

\subsection{The main result}\label{sec:main-result}
{Let us first summarise the notation we use throughout this article. For operators we use upper case calligraphic letters, such as $\oper{T}$. The spectrum of $\oper{T}$ is denoted $\mathrm{sp}(\oper{T})$. For the sesquilinear form associated to an operator we use the same letter in lower case Fraktur font. Hence the operator $\oper{T}$ has the associated form $\sform{t}$. The space of bounded linear operators on a Hilbert space $\mathfrak{H}$ is denoted $\mathfrak{B}(\mathfrak{H})$. Domains of operators or forms are denoted by $\mathfrak{D}$. The graph norms of an operator $\oper{T}$ and a form $\sform{T}$ are denoted $\norm{\cdot}_{\oper{T}}$ and $\norm{\cdot}_{\sform{T}}$, respectively. Strong, strong resolvent and norm resolvent convergence are denoted by $\tos$, $\tosr$ and $\tonr$, respectively.  For brevity, we denote $\overline{\mathbb{N}}=\mathbb{N}\cup\{\infty\}$. We also recall the definition of a sectorial form:

\begin{definition}
A form $\sform{t}$ is said to be \emph{sectorial} if its numerical range $\Theta(\sform{t})$ (that is, the set $\{\sform{t}[u,u] : \|u\|=1,\, u\in\dom{\sform{t}}\}\subseteq\mathbb{C}$) is a subset of a sector of the form
	\begin{equation*}
	\left\{\zeta : |\arg(\zeta-\gamma)|\leq\theta\right\},\qquad\theta\in[0,\pi/2),\quad\gamma\in\mathbb{R}.
	\end{equation*}

\end{definition}

}

Let $\mathfrak{H}=\mathfrak{H}_+\oplus\mathfrak{H}_-$ be a (separable) Hilbert space with inner product $\ip{\cdot}{\cdot}$ and norm $\norm{\cdot}$ and let 
	\begin{equation*}
	\oper{A}^\lambda=
	\begin{bmatrix}
	\mathcal{A}_+^\lambda&0\\
	0&-\mathcal{A}_-^\lambda
	\end{bmatrix}
	\quad\text{and}\quad
	\oper{K}^\lambda=
	\begin{bmatrix}
	\oper{K}^\lambda_{++}&\oper{K}^\lambda_{+-}\\
	\oper{K}^\lambda_{-+}&\oper{K}^\lambda_{--}
	\end{bmatrix},\quad\lambda\in[0,1]
	\end{equation*}
be two families of operators on $\mathfrak{H}$ depending upon the parameter $\lambda\in[0,1]$, where the family $\oper{A}^\lambda$ is also assumed to be defined for $\lambda$ in an open neighbourhood $D$ of $[0,1]$ in the complex plane. {The two families $\oper{A}^\lambda$ and $\oper{K}^\lambda$ satisfy}:\\

\textbf{i) Sectoriality:} The {families $\{\oper{A}_\pm^\lambda\}_{\lambda\in D}$ are \emph{holomorphic of type (B)}\footnote{We adopt the terminology of Kato \cite{Kato1995}.}. That is, they are families of sectorial operators and the associated sesquilinear  forms $\sform{a}_\pm^\lambda$ are \emph{holomorphic of type (a)}: all $\{\sform{a}_\pm^\lambda\}_{\lambda\in D}$ are sectorial and closed, with domains that are independent of $\lambda$ and dense in $\spac{H}_\pm$,\footnote{Hence we shall remove the $\lambda$ superscript when discussing the domains of $\sform{a}^\lambda$ and $\sform{a}_\pm^\lambda$.}  and $D\ni\lambda\mapsto\sform{a}^\lambda_\pm[u,v]$ are holomorphic for any $u,v\in\dom{\sform{a}^\lambda_\pm}$.} Furthermore, we assume that $\oper{A}^\lambda_\pm$ are self-adjoint for $\lambda\in[0,1]$.\\

\textbf{ii) Gap:} $\oper{A}^\lambda_\pm>1$ for every $\lambda\in[0,1]$.\\

\textbf{iii) Bounded perturbation:} $\{\oper{K}^\lambda\}_{\lambda\in[0,1]}\subset\bounded{\mathfrak{H}}$ is a {self-adjoint} strongly continuous family.\\

\textbf{iv) Compactness:} There exist {self-adjoint} operators $\oper{P}_\pm\in\bounded{\mathfrak{H}_\pm}$ which are relatively compact with {respect to $\oper{A}^\lambda_\pm$}, satisfying $\oper{K}^\lambda=\oper{K}^\lambda\oper{P}$ for all $\lambda\in[0,1]$ where \[\oper{P}=\begin{bmatrix}\oper{P}_+&0\\0&\oper{P}_-\end{bmatrix}.\]\\

Finally, if the family $\oper{A}^\lambda$ does not have a compact resolvent we assume:\\

\textbf{v) Compactification of the resolvent:}
There exist holomorphic forms $\{\sform{w}_\pm^\lambda\}_{\lambda\in D}$ of type (a) and associated operators $\{\oper{W}^\lambda_\pm\}_{\lambda\in D}$ of type (B) such that  for $\lambda\in[0,1]$, $\oper{W}^\lambda_\pm$ are self-adjoint and non-negative. {Define  	 \[\oper{W^\lambda}=\begin{bmatrix}\oper{W}_+^\lambda&0\\0&-\oper{W}_-^\lambda\end{bmatrix},\quad \lambda\in D,\] 
and
\begin{equation}\label{eq:a-epsilon}
	\oper{A}_\varepsilon^\lambda:=\oper{A}^\lambda+\varepsilon\oper{W}^\lambda,\quad \lambda\in D,\ \varepsilon\ge0
	\end{equation}
with  respective associated forms $\sform{W}^\lambda$ and $\sform{A}_\varepsilon^\lambda$}.
Then we assume that $\dom{\sform{W}^\lambda}\cap\dom{\sform{A}}$ are dense for all $\lambda\in D$ and the inclusion $(\dom{\sform{W}^\lambda}\cap\dom{\sform{A}},\norm{\cdot}_{\sform{A}_\varepsilon^\lambda})\to(\mathfrak{H},\norm{\cdot})$ is compact for some $\lambda\in D$ and all $\varepsilon>0$.\\

\paragraph{Goal.}
Define the family of (unbounded) operators $\{\oper{M}^\lambda\}_{\lambda\in[0,1]}$, acting in $\mathfrak{H}$,  as
	\begin{equation}\label{Afamilydef}
	\oper{M}^\lambda=\oper{A}^\lambda+\oper{K}^\lambda,\quad \lambda\in [0,1].
	\end{equation}
It is these operators that we wish to approximate.

\paragraph{The Projections.} Let $\oper{A}_\varepsilon^\lambda$ be as in \eqref{eq:a-epsilon}, and define
	\begin{equation}\label{eq:m-epsilon}
	\mathcal{M}^\lambda_\varepsilon=\mathcal{A}^\lambda_\varepsilon+\mathcal{K}^\lambda,\quad\lambda\in[0,1].
	\end{equation}
Let	
\begin{itemize}
\item $\{e^\lambda_{\varepsilon,k}\}_{k\in\mathbb{N}}\subset\mathfrak{H}$ be a complete orthonormal set of eigenfunctions of $\mathcal{A}^\lambda_\varepsilon$,
\item $\mathcal{G}^\lambda_{\varepsilon,n}:\mathfrak{H}\to\mathfrak{H}$ be the orthogonal projection operators onto $\mathrm{span}(e^\lambda_{\varepsilon,1},\dots,e^\lambda_{\varepsilon,n})$,
\item $\widetilde{\oper{M}}^{\lambda}_{\varepsilon,n}$ be the $n$-dimensional operator defined as the restriction of ${\oper{M}}^{\lambda}_{\varepsilon}$ to $\mathcal{G}^\lambda_{\varepsilon,n}(\mathfrak{H})$.
\end{itemize}

Fix $\varepsilon^*>0$, and define the function
\begin{equation*}
\begin{aligned}
 \Sigma:[0,1]\times[0,\varepsilon^*]&\to(\text{closed subsets of } (-1,1),d_H)\\
\Sigma(\lambda,\varepsilon)&=(-1,1)\cap\mathrm{sp}(\oper{M}^{\lambda}_{\varepsilon})\\
\end{aligned}
\end{equation*}
and for fixed $\varepsilon>0$
\begin{equation*}
\begin{aligned}
 \Sigma_{\varepsilon}:[0,1]\times\mathbb{N}&\to(\text{closed subsets of } (-1,1),d_H)\\
\Sigma_{\varepsilon}(\lambda,n)&=(-1,1)\cap\mathrm{sp}(\widetilde{\oper{M}}^{\lambda}_{\varepsilon,n})
\end{aligned}
\end{equation*}
where $\mathrm{sp}(\mathcal{O})$ is the spectrum of the operator $\mathcal{O}$  and $d_H$ is the Hausdorff distance, defined for two bounded sets $X,Y\subset\mathbb{C}$ as:
	\[
	d_H(X,Y)=\max\left(\sup_{y\in Y}\inf_{x\in X}|x-y|,\sup_{x\in X}\inf_{y\in Y}|x-y|\right).
	\]
{This defines a pseudometric, which becomes a metric if restricting to \emph{closed} bounded sets (this is indeed the case here, see \autoref{remark:topology} below)}. Our main result is formulated for the general case where the spectrum of $\oper{A}^\lambda$ may have a continuous part:

\begin{theorem}\label{thm:main}
The mappings $\Sigma(\cdot,\cdot)$ and $\Sigma_{\varepsilon}(\cdot,n)$ are continuous in their arguments, and as $n\to\infty$, $\Sigma_{\varepsilon}(\lambda,n)\to \Sigma(\lambda,\varepsilon)$ uniformly in $\lambda\in[0,1]$.
\end{theorem}
{\begin{remark}\label{remark:topology}
It is well known that the spectrum of an operator is a closed set. Moreover, in our case we know that the spectrum in $(-1,1)$ is discrete and with no accumulation points. Hence when it is stated that $\Sigma$ and $\Sigma_\varepsilon$ take values in ``closed subsets of  $(-1,1)$''  there is no ambiguity with respect to which topology is considered: the standard topology on the real line, or the topology on $(-1,1)$ thought of as a subspace of the real line. We consider the standard topology on the real line.
\end{remark}}

\paragraph{A Simpler Case: Semi-Bounded Operators.} As the notation becomes quite cumbersome due to the decomposition $\mathfrak{H}=\mathfrak{H}_+\oplus\mathfrak{H}_-$, we shall first treat the simpler case of semi-bounded operators. {Let $\oper{A}^\lambda$ and $\oper{K}^\lambda$, where $\lambda\in[0,1]$, be two families of operators on some Hilbert space $\mathfrak{H}$  (which is not assumed to decompose as before) where the family $\oper{A}^\lambda$ is also assumed to be defined for $\lambda$ in an open neighbourhood $D$ of $[0,1]$ in the complex plane. For the sake of precision, we repeat the assumptions (i)-(v) reformulated for this case.

\textbf{i) Sectoriality:} The family $\oper{A}^\lambda$ is sectorial of type (B) in $\lambda\in D$ and self-adjoint for $\lambda\in[0,1]$.

\textbf{ii) Semi-boundedness:} $\oper{A}^\lambda>1$ for every $\lambda\in[0,1]$.

\textbf{iii) Bounded perturbation:} $\{\oper{K}^\lambda\}_{\lambda\in[0,1]}\subset\bounded{\mathfrak{H}}$ is a {self-adjoint} strongly continuous family.

\textbf{iv) Compactness:} There exists a {self-adjoint} operator $\oper{P}\in\bounded{\mathfrak{H}}$ which is relatively compact with respect to $\oper{A}^\lambda$, satisfying $\oper{K}^\lambda=\oper{K}^\lambda\oper{P}$ for all $\lambda\in[0,1]$.

\textbf{v) Compactification of the resolvent:}
There exist  holomorphic forms $\{\sform{w}^\lambda\}_{\lambda\in D}$ of type (a) and associated operators $\{\oper{W}^\lambda\}_{\lambda\in D}$ of type (B) such that  for $\lambda\in[0,1]$, $\oper{W}^\lambda$ are self-adjoint and non-negative. {Define  
\begin{equation*}
	\oper{A}_\varepsilon^\lambda:=\oper{A}^\lambda+\varepsilon\oper{W}^\lambda,\quad \lambda\in D,\ \varepsilon\ge0
	\end{equation*}
with  respective associated forms $\sform{W}^\lambda$ and $\sform{A}_\varepsilon^\lambda$}.
Then we assume that $\dom{\sform{W}^\lambda}\cap\dom{\sform{A}}$ are dense for all $\lambda\in D$ and the inclusion $(\dom{\sform{W}^\lambda}\cap\dom{\sform{A}},\norm{\cdot}_{\sform{A}_\varepsilon^\lambda})\to(\mathfrak{H},\norm{\cdot})$ is compact for some $\lambda\in D$ and all $\varepsilon>0$.}\\

We define the projections as above and therefore do not repeat the definition again. However, we do define the functions $\Sigma$ and $\Sigma_{\varepsilon}$ again\footnote{Despite the slight abuse of notation, we do not alter the names $\Sigma$ and $\Sigma_{\varepsilon,n}$} as their ranges are now different. Now fix $\varepsilon^*>0$, and define the function
\begin{equation*}
\begin{aligned}
 \Sigma:[0,1]\times[0,\varepsilon^*]&\to(\text{closed bounded subsets of } (-\infty,1),d_H)\\
\Sigma(\lambda,\varepsilon)&=(-\infty,1)\cap\mathrm{sp}(\oper{M}^{\lambda}_{\varepsilon})\\
\end{aligned}
\end{equation*}
and for fixed $\varepsilon>0$ and $n\in\mathbb{N}$ the function
\begin{equation*}
\begin{aligned}
 \Sigma_{\varepsilon}:[0,1]\times\mathbb{N}&\to(\text{closed bounded subsets of } (-\infty,1),d_H)\\
\Sigma_{\varepsilon}(\lambda,n)&=(-\infty,1)\cap\mathrm{sp}(\widetilde{\oper{M}}^{\lambda}_{\varepsilon,n}).\\
\end{aligned}
\end{equation*}

\setcounter{theorem}{3}
 \begin{thmbis}\label{thm:main2}
In the semi-bounded case the mappings $\Sigma(\cdot,\cdot)$ and $\Sigma_{\varepsilon}(\cdot,n)$ are also continuous in their arguments, and as $n\to\infty$, $\Sigma_{\varepsilon}(\lambda,n)\to \Sigma(\lambda,\varepsilon)$ uniformly in $\lambda\in[0,1]$.
\end{thmbis}
\setcounter{theorem}{4}

In the subsequent sections we will prove \autoref{thm:main2} before proving \autoref{thm:main} in \autoref{sec:nonpositive}.
{\begin{remark}
As in \autoref{remark:topology}, here too the spectrum in $(-\infty,1)$ is discrete (with no accumulation points) so that there is no topological ambiguity when stating that a set is ``closed''. Note that as the operators $\oper{M}_\varepsilon^\lambda$ are semi-bounded the sets in question are indeed bounded. Hence, when restricted to these sets, the Hausdorff distance defines a metric.

Thus an immediate corollary of both theorems, by the Heine-Cantor theorem, is that the two maps $\Sigma(\cdot,\cdot)$ and $\Sigma_\varepsilon(\cdot,n)$  are in fact uniformly continuous.
\end{remark}}

\subsection{Discussion}
One of the main driving forces behind the study of linear operators in the 20th century was the development of quantum mechanics. Particular attention had been given to the characterisation of the spectra of such operators, as it encodes many important physical properties (such as energy levels, for instance). When operators become too complex, a typical approach is to view them as perturbations of simpler operators whose spectrum is well understood. Two of the classic texts on this topic are those written by Kato \cite{Kato1995} and Reed and Simon \cite{Reed1978}. Both are still widely cited to this day. We also refer to Simon's review paper \cite{Simon1991} and the references therein.

Recently, Hansen \cite{Hansen2008a} presented new techniques for approximating spectra of linear operators (self-adjoint and non-self-adjoint) from a more computational point of view. In \cite{Strauss2014}, Strauss presents a new method for approximating eigenvalues and eigenvectors of self-adjoint operators via an algorithm that is itself self-adjoint, and which does not produce spectral pollution. Both papers provide extensive references to additional literature in the field. We also mention  \cite{Kumar2014}, where  analysis similar to ours is performed for bounded operators. We note that spectral pollution (the appearance of spurious eigenvalues within gaps in the essential spectrum when approximating) has attracted significant attention \cite{Davies2004a,Levitin2004,Lewin2009}. {We do not encounter this issue here because of how the problem is set up: the trial spaces are (and therefore commute with) the spectral projectors of the block diagonal parts of the unperturbed operator, see e.g. \cite{Lewin2009} for more discussion of this topic.}

The question that we are motivated by is somewhat different. We are interested in the simultaneous approximation of \emph{families of operators}, rather than approximating a single fixed linear operator. This may be viewed as perturbation theory with two parameters: the continuous parameter $\lambda$ representing small continuous perturbations generating the family of operators, and the discrete parameter $n$ representing the dimension of the finite-dimensional approximation. One of the important aspects of this theory is that the finite-dimensional approximations {approximate  the entire family of operators uniformly in $\lambda$}. Previously, in \cite[Proposition 2.5]{Ben-Artzi2011b} a much weaker result of this type was obtained, where the resolvent set of Schr\"odinger operators with a compact resolvent was shown to be stable under similar perturbations. {We also mention \cite{Deconinck2006,Curtis2010,Johnson2012} where the convergence of the so-called \emph{Hill's method} (or Fourier-Floquet-Hill) is studied. This is a numerically-oriented method for studying spectra of periodic differential operators (not necessarily self-adjoint) and involves the truncation of the associated Fourier series. We refer in particular to \cite{Johnson2012} for an instance where this method is also applied to a \emph{family} of operators.}

There are two substantial difficulties in proving our results. If the spectrum of $\oper{A}^\lambda$ were discrete for some $\lambda$ (and therefore for all $\lambda$) we would have a natural way to construct approximations by projecting onto increasing subspaces associated to the eigenvalues of $\oper{M}^\lambda$. However we do not require the spectrum to be discrete, and, indeed, in the type of problems we have in mind it is not. This necessitates the introduction of yet another perturbation parameter, $\varepsilon$, related to the compactification of the resolvent. The other difficulty is in ensuring that the finite-dimensional approximations approximate the whole family of operators {uniformly in $\lambda$}. To this end, the compactness assumption (iv) plays a crucial role.

We make several remarks on \autoref{thm:main} and \autoref{thm:main2} and the assumptions (i)-(v):

\begin{remark}
The compactness requirements (iv) on $\oper{P}$ are motivated by \eqref{eq:non-positive}. If $\oper{A}$ has a compact resolvent (e.g. when acting in $L^2(\mathbb{T}^d)\oplus L^2(\mathbb{T}^d)$ where $\mathbb{T}^d$ is the $d$-dimensional torus)  we may take $\oper{P}$ to be the identity. Otherwise (e.g. for $L^2(\mathbb{R}^d)\oplus L^2(\mathbb{R}^d)$) if the perturbations $\oper{K}^\lambda$ are compactly supported in the sense that
\begin{equation}\label{Ksupportassumption}
 \bigcup_{\lambda\in[0,1],u\in\mathfrak{H}}\supp({\oper{K}^\lambda u})\subset K
\end{equation}
where $K=K_+\times K_-\subset\mathbb{R}^d\times\mathbb{R}^d$ is compact, then we may take $\oper{P}_\pm$ as multiplications by the indicator functions of the sets $K_\pm$. Indeed, we first note that \eqref{Ksupportassumption} implies that for all $\lambda$, $\oper{K}^\lambda=\oper{P}\oper{K}^\lambda$. Then as $\oper{K}^\lambda$ and $\oper{P}$ are symmetric, we deduce that $\oper{K}^\lambda=(\oper{K}^\lambda)^*=(\oper{K}^\lambda)^*\oper{P}^*=\oper{K}^\lambda\oper{P}$ as required. That $\oper{P}$ is relatively compact with respect to $-\Delta$ follows from Rellich's theorem. We also remark that this choice of $\oper{P}$ is in fact the natural inclusion map from $L^2$ to $L^2(K)$.
\end{remark}

\begin{remark}\label{rek:thm}
Care must be taken regarding the spaces we view operators as acting on. If we view $\oper{M}_{\varepsilon,n}^\lambda=\mathcal{G}^\lambda_{\varepsilon,n}\mathcal{M}^\lambda_\varepsilon\mathcal{G}^\lambda_{\varepsilon,n}:\mathfrak{H}\to\mathfrak{H}$ then $0$ will \emph{always} be a spurious eigenvalue with infinite multiplicity.
To remove this unwanted eigenvalue we must instead consider $\widetilde{\oper{M}}_{\varepsilon,n}^\lambda:\mathfrak{H}^\lambda_{\varepsilon,n}\to\mathfrak{H}^\lambda_{\varepsilon,n}$ where $\mathfrak{H}^\lambda_{\varepsilon,n}=\oper{G}^\lambda_{\varepsilon,n}(\mathfrak{H})$ is the $n$-dimensional space corresponding to the eigenprojection $\oper{G}^\lambda_{\varepsilon,n}$.
\end{remark}

\begin{remark}\label{rek:alpha}
Property (ii) implies that  there exists $\alpha(\lambda)>0$ such that $(-\alpha(\lambda)-1,1+\alpha(\lambda))$ is in the resolvent set of $\oper{A}^\lambda$. Since the spectrum is continuous in $\lambda\in[0,1]$ this implies that there is a uniform constant $\alpha>0$ such that $(-\alpha-1,1+\alpha)$ is in the resolvent set of $\oper{A}^\lambda$ for all $\lambda\in[0,1]$.
\end{remark}

{\begin{remark}
We finally remark that the construction of a compactifying  operator $\oper{W}$ in general is not easy. We have in mind an application to a case where this is applied to $-\Delta$ and then it is simple: any unbounded potential will do.
\end{remark}}

This paper is organised as follows. In \autoref{sec:preliminary} we present some results related to general properties (such as self-adjointness, equivalence of norms, etc.) of the various operators. In \autoref{sec:approx} we construct the finite-dimensional approximations to our family of operators, which are used in \autoref{sec:proof} to prove \autoref{thm:main2}. In \autoref{sec:nonpositive} these results are extended to families of operators which are not positive, proving \autoref{thm:main}. Finally, in \autoref{sec:application} we give a brief description of an application of these results related to plasma instabilities, which is the subject of \cite{Ben-Artzi2013a} where one can find the full details.

\section{Preliminary results}\label{sec:preliminary}
\emph{We remind the reader that in this section, as well as in \autoref{sec:approx} and \autoref{sec:proof} we treat the semi-bounded case (\autoref{thm:main2}).} 

\bigskip

Considering the definition \eqref{Afamilydef} and the subsequent specifications of the properties of the various operators and associated forms, we have the following results.
{\begin{lemma}\label{lem:basic-properties1}
The forms $\sform{m}^\lambda$  have  the same domains as the forms $\sform{a}^\lambda$, and  are independent of $\lambda$. For any $\lambda\in[0,1]$, $\oper{M}^\lambda$ is self-adjoint and has the same essential spectrum and domain as $\oper{A}^\lambda$. In particular its spectrum inside $(-\infty,1]$ is discrete.
\end{lemma}
\begin{proof}
The equality $\dom{\sform{M}^\lambda}=\dom{\sform{A}^\lambda}$ holds since $\oper{K}^\lambda$ is bounded for each $\lambda$. The fact that the domains are independent of $\lambda$ was assumed above in the sectoriality assumption (i). Self-adjointness follows from the Kato-Rellich theorem, due to $\oper{A}^\lambda$ being self-adjoint for $\lambda\in[0,1]$ and the symmetry assumption (iii) on $\oper{K}^\lambda$. The essential spectrum result follows from Weyl's theorem as $\oper{K^\lambda}=\oper{K^\lambda}\oper{P}$ is relatively compact with respect to $\oper{A}^\lambda$ (for any $\lambda$) because $\oper{P}$ is.
\end{proof}}

Next, we turn our attention to the map $\lambda\mapsto\oper{M}^\lambda$. Intuitively, one would expect $\oper{M}^\lambda$ to have continuity properties similar to those of $\oper{K}^\lambda$ and therefore be merely continuous in the strong resolvent sense. In fact, due to the relative compactness assumption on $\oper{P}$ we have more:
\begin{proposition}\label{lem:cont-alambda}
The family $\{\oper{M}^\lambda\}_{\lambda\in[0,1]}$ is norm resolvent continuous. 
\end{proposition}
\begin{proof}
Fix some $\lambda\in[0,1]$ and let $[0,1]\ni\lambda_n\to\lambda$ as $n\to\infty$. It is sufficient to prove 
\begin{equation*}
 \norm{(\oper{M}^{\lambda_n}+i)^{-1}-(\oper{M}^\lambda+i)^{-1}}_{\bounded{\mathfrak{H}}}\to0 \text{ as }n\to\infty.
\end{equation*}
Using the triangle inequality we have
\begin{equation*}
\begin{aligned}
\norm{(\oper{M}^{\lambda_n}+i)^{-1}-(\oper{M}^\lambda+i)^{-1}}_{\bounded{\mathfrak{H}}}&\le\norm{(\oper{M}^{\lambda_n}+i)^{-1}-(\oper{A}^{\lambda_n}+\oper{K}^\lambda+i)^{-1}}_{\bounded{\mathfrak{H}}}\\
&\quad+\norm{(\oper{A}^{\lambda_n}+\oper{K}^\lambda+i)^{-1}-(\oper{M}^{\lambda}+i)^{-1}}_{\bounded{\mathfrak{H}}}.
\end{aligned}
\end{equation*}
By observing that $\{\oper{A}^\sigma+\oper{K}^\lambda\}_{\sigma\in D}$ is also a holomorphic family of type (B) we deduce that the second term tends to zero as $n\to\infty$. For the first term we follow the method used to deduce the second Neumman series (see \cite[II-(1.13)]{Kato1995})
\begin{equation*}
(\oper{A}^{\lambda_n}+\oper{K}^{\lambda_n}+i)^{-1}=(\oper{A}^{\lambda_n}+\oper{K}^\lambda+i)^{-1}(1+(\oper{K}^{\lambda_n}-\oper{K}^\lambda)(\oper{A}^{\lambda_n}+\oper{K}^{\lambda}+i)^{-1})^{-1}
\end{equation*}
which is valid whenever $\norm{(\oper{K}^{\lambda_n}-\oper{K}^\lambda)(\oper{A}^{\lambda_n}+\oper{K}^{\lambda}+i)^{-1}}_{\bounded{\mathfrak{H}}}<1$. By the norm resolvent continuity of operator inversion and again using the norm resolvent continuity of the family $\{\oper{A}^\sigma+\oper{K}^\lambda\}_{\sigma\in [0,1]}$, it is sufficient to show that
\begin{equation}\label{eq:sufficient to show that}
\norm{(\oper{K}^{\lambda_n}-\oper{K}^\lambda)(\oper{A}^{\lambda}+\oper{K}^{\lambda}+i)^{-1}}_{\bounded{\mathfrak{H}}}\to0\text{ as }n\to\infty.
\end{equation}
We observe that $\oper{A}^\lambda+\oper{K}^\lambda$ is self-adjoint with the same domain as $\oper{A}^\lambda$ by \autoref{lem:basic-properties1}, so $\oper{P}$ is also relatively compact with respect to $\oper{A}^\lambda+\oper{K}^\lambda$. By assumption (iv) we have
\begin{equation*}
(\oper{K}^{\lambda_n}-\oper{K}^\lambda)(\oper{A}^{\lambda}+\oper{K}^{\lambda}+i)^{-1}=(\oper{K}^{\lambda_n}-\oper{K}^\lambda)\oper{P}(\oper{A}^{\lambda}+\oper{K}^{\lambda}+i)^{-1}.
\end{equation*}
This is a composition of a strongly convergent sequence of operators and the compact operator $\oper{P}(\oper{A}^{\lambda}+\oper{K}^{\lambda}+i)^{-1}$. The compactness converts the strong convergence to norm convergence and proves \eqref{eq:sufficient to show that}.
\end{proof}

\section{Constructing approximations}\label{sec:approx}
We first treat approximations of operators with discrete spectra, which are naturally defined via a sequence of increasing projection operators. For brevity, we call these approximations \emph{$n$-approximations} (``$n$'' refers to the dimension of the projection). Then, our strategy when treating operators with a \emph{continuous} spectrum is to first ``perturb'' them by adding a family of unbounded operators (think of adding an unbounded potential to a Laplacian) depending upon a small parameter $\varepsilon$. For each $\varepsilon>0$ these perturbations are assumed to eliminate any continuous spectrum, so that then we may apply an $n$-approximation. We therefore call these \emph{$(\varepsilon,n)$-approximations}.
We start with a standard result for which we could not find a good reference and we therefore state and prove it here.
\begin{lemma}\label{lem:appendix1}
Let $\spac{H}$ be a Hilbert space and let $\oper{T}_n\tosr \oper{T}$ as $n\to\infty$ with $\oper{T}_n,\oper{T}$ self-adjoint operators on $\spac{H}$. Let $\oper{K}_n\tos\oper{K}$ as $n\to\infty$ with $\oper{K}_n,\oper{K}$ bounded {self-adjoint} operators on $\spac{H}$. Then $\oper{T}_n+\oper{K}_n$ and $\oper{T}+\oper{K}$ are self-adjoint {in $\spac{H}$} and $\oper{T}_n+\oper{K}_n\tosr\oper{T}+\oper{K}$.
\end{lemma}
\begin{proof}
The self-adjointness follows from the Kato-Rellich theorem. For the convergence it is sufficient to prove that $(\oper{T}_n+\oper{K}_n+\alpha i)^{-1}\tos(\oper{T}+\oper{K}+\alpha i)^{-1}$ for some real $\alpha\ne0$. As the $\oper{K}_n$ are strongly convergent, by the uniform boundedness principle they are uniformly bounded in operator norm by some $M\ge\norm{\oper{K}}_{\bounded{\spac{H}}}$. Letting $\alpha=2M$, and using the second Neumann series,
\begin{equation*}
\begin{aligned}
(\oper{T}_n+\oper{K}_n+\alpha i)^{-1}&=(\oper{T}_n+\alpha i)^{-1}(1+\oper{K}_n(\oper{T}_n+\alpha i)^{-1})^{-1}\\
&=(\oper{T}_n+\alpha i)^{-1}\sum_{k=0}^\infty(-1)^k(\oper{K}_n(\oper{T}_n+\alpha i)^{-1})^k
\end{aligned}
\end{equation*}
is convergent uniformly in $n$ as $\norm{\oper{K}_n(\oper{T}_n+\alpha i)^{-1}}_{\bounded{\spac{H}}}\le M/\alpha=1/2<1$.  As $n\to\infty$ each term of the series converges strongly to the corresponding term of the series for $(\oper{T}+\oper{K}+\alpha i)^{-1}$ and as the series convergences uniformly in $n$ we may may swap the order of summation and {take} strong limits.
\end{proof}

\subsection{Operators with discrete spectra}
In this paragraph we assume that $\oper{A}^\lambda$ has discrete spectrum and compact resolvent for some $\lambda$ (and, in fact, for all $\lambda$, as $\oper{A}^\lambda$ is a holomorphic family of type (B)\footnote{See property (i) in \autoref{sec:main-result} for a precise definition.}). We exploit a property of self-adjoint holomorphic families \cite[VII Theorem 3.9 and VII Remark 4.22]{Kato1995}: {all eigenvalues of $\oper{A}^\lambda$ can be represented by functions which are holomorphic on $[0,1]$. That is, there exists a sequence of scalar-valued functions $\{\mu_k^\lambda\}_{k\in\mathbb{N}}$ which are all holomorphic functions of  $\lambda\in[0,1]$ that represents all the repeated eigenvalues of $\oper{A}^\lambda$. Moreover, there exists  a sequence of vector-valued functions $\{e_k^\lambda\}_{k\in\mathbb{N}}$ which are all also holomorphic functions of $\lambda\in[0,1]$ such that for every $\lambda\in[0,1]$, $\{e_k^\lambda\}_{k\in\mathbb{N}}$ form a complete orthonormal family of corresponding eigenvectors.} An immediate consequence is that the unitary operator defined by
\begin{equation*}
\begin{aligned}
\oper{U}^\lambda_\sigma&:\mathfrak{H}\to\mathfrak{H}\\
e_k^\sigma&\mapsto e_k^\lambda&\text{ for any }k\in\mathbb{N}
\end{aligned}
\end{equation*}
is jointly holomorphic in $\lambda,\sigma\in[0,1]${, i.e. possesses  a locally convergent power series in the two variables $\lambda,\sigma$}. We now define the $n$-truncation operator by
\begin{equation*}
\begin{aligned}
\oper{G}^\lambda_n&:\mathfrak{H}\to\mathfrak{H}\\
e_k^\lambda&\mapsto\begin{cases}
e_k^\lambda&\text{ if }k\le n,\\
0&\text{ if }k>n.
\end{cases}
\end{aligned}
\end{equation*}
Since the eigenfunctions form a complete orthonormal set we have the convergence $\oper{G}^\lambda_n\tos1$ as $n\to\infty$ for fixed $\lambda$. Additionally by expressing $\oper{G}^\lambda_n=\oper{U}^\lambda_\sigma\oper{G}^\sigma_n\oper{U}^\sigma_\lambda$ for some fixed $\sigma\in[0,1]$ we see that $\oper{G}^\lambda_n\tos 1$ as $n\to\infty$. Moreover,  for any sequence $\lambda_n\to\lambda$ we have $\oper{G}^{\lambda_n}_n\tos 1$ as $n\to\infty$. For notational convenience we define $\oper{G}^\lambda_\infty=1$ for all $\lambda\in[0,1]$.

We now define the finite-dimensional approximations of  $\oper{A}^\lambda$ and $\oper{M}^\lambda$ by
\begin{equation}\label{Andef}
\oper{A}^\lambda_n=\oper{G}^\lambda_n\oper{A}^\lambda\oper{G}^\lambda_n\quad\text{and}\quad\oper{M}^\lambda_n=\oper{G}^\lambda_n\oper{M}^\lambda\oper{G}^\lambda_n,
\end{equation}
respectively. 
It is too much to hope for convergence $\oper{M}_n^\lambda\tonr\oper{M}^\lambda$ as $n\to\infty$, but we can hope for  $\oper{M}_n^\lambda\tosr\oper{M}^\lambda$. Indeed:
\begin{lemma}\label{src of truncation}
For any sequence $\lambda_n\to\lambda\in[0,1]$ as $n\to\infty$, we have the convergence $\oper{M}^{\lambda_n}_n\tosr \oper{M}^\lambda$.
\end{lemma}
\begin{proof}

By the stability of strong resolvent continuity with respect to bounded strongly continuous perturbations (see \autoref{lem:appendix1}), it is sufficient to prove that $\oper{A}^{\lambda_n}_{n}\tosr\oper{A}^\lambda$ as $n\to\infty$ and that $\oper{G}^{\lambda_n}_{n}\oper{K}^{\lambda_n}\oper{G}^{\lambda_n}_{n}\tos\oper{K}^\lambda$. The latter is true as it is the composition of strong convergences  of bounded operators. For the former it is sufficient to show that $(\oper{A}^{\lambda_n}_{n}+i)^{-1}\tos(\oper{A}^\lambda+i)^{-1}$ as $n\to\infty$. Splitting this term as
\begin{equation*}
(\oper{A}^{\lambda_n}_{n}+i)^{-1}=\oper{G}^{\lambda_n}_{n}(\oper{A}^{\lambda_n}_{n}+i)^{-1}\oper{G}^{\lambda_n}_{n}+(1-\oper{G}^{\lambda_n}_{n})(\oper{A}^{\lambda_n}_{n}+i)^{-1}(1-\oper{G}^{\lambda_n}_{n}),
\end{equation*}
(where we have used the fact that $\oper{G}^{\lambda_n}_{n}$ is a spectral projection which commutes with $(\oper{A}^{\lambda_n}_{n}+i)^{-1}$),
we see that the second term converges strongly to zero since $(\oper{A}^{\lambda_n}_{n}+i)^{-1}$ is uniformly bounded and since $\oper{G}^{\lambda_n}_{n}\tos1$. For the first term on the right hand side, note that
\begin{equation*}
\oper{G}^{\lambda_n}_{n}(\oper{A}^{\lambda_n}_{n}+i)^{-1}\oper{G}^{\lambda_n}_{n}=\oper{G}^{\lambda_n}_{n}(\oper{A}^{\lambda_n}+i)^{-1}\oper{G}^{\lambda_n}_{n}
\end{equation*}
which converges strongly to $(\oper{A}^\lambda+i)^{-1}$ by the composition of strong convergences.
\end{proof}

\subsection{Operators with continuous spectra}
We are now ready to turn to the general case of families $\{\oper{A}^\lambda\}_{\lambda\in[0,1]}$ that may have continuous spectra. Such operators require $(\varepsilon,n)$-approximations. The $\varepsilon$-approximations $\oper{A}^\lambda_\varepsilon$ of $\oper{A}^\lambda$ were defined in \eqref{eq:a-epsilon} and the corresponding approximations $\mathcal{M}^\lambda_\varepsilon$ were defined in \eqref{eq:m-epsilon}.
\begin{lemma}
\begin{enumerate}
 \item For any $\varepsilon>0$, $\{\oper{A}^\lambda_\varepsilon\}_{\lambda\in D}$ is a holomorphic family of type (B) with compact resolvent.
\item For any $\lambda\in[0,1],\varepsilon\ge0$, $\oper{A}_\varepsilon^\lambda$ is self-adjoint and we have $\oper{A}_\varepsilon^\lambda\ge\oper{A}^\lambda\ge1+\alpha$, where $\alpha$ was defined in \autoref{rek:alpha}.
\end{enumerate}
\end{lemma}
\begin{proof}
The second claim is obvious since $\oper{W}^\lambda\ge0$. For the first we must show that $\sform{A}_\varepsilon^\lambda$ is sectorial and that its domain $\dom{\sform{A}_\varepsilon^\lambda}$ is independent of $\lambda$ and dense in $\mathfrak{H}$, and that for any fixed $u\in\dom{\sform{A}_\varepsilon^\lambda}$ the function $\sform{A}^\lambda_\varepsilon[u]$ is holomorphic in $\lambda\in D$. For any $\lambda\in D$, $\sform{A}_\varepsilon^\lambda$ is the sum of the sectorial forms $\sform{a}^\lambda$ and $\varepsilon \sform{W}^\lambda$ so by \cite[VI\S1.6-Theorem 1.33]{Kato1995} it is closed and sectorial with domain $\dom{\sform{A}}\cap\dom{\sform{W}^\lambda}$, which is independent of $\lambda$ since both $\oper{A}^\lambda$ and $\oper{W}^\lambda$ are holomorphic families of type (B). Furthermore, we assumed that $\dom{\sform{A}}\cap\dom{\sform{W}^\lambda}$ is dense in $\mathfrak{H}$. For any fixed $u\in\dom{\sform{A}_\varepsilon^\lambda}$, $\sform{A}^\lambda_\varepsilon[u]=\sform{A}^\lambda[u]+\varepsilon\sform{W}^\lambda[u]$ is the sum of two holomorphic functions of $\lambda\in D$, so $\sform{A}^\lambda_\varepsilon[u]$ is also holomorphic in $D$. Finally by the assumption that the inclusion  $(\dom{\sform{A}_\varepsilon^\lambda},\norm{\cdot}_{\sform{A}_\varepsilon^\lambda})\hookrightarrow\mathfrak{H}$ is compact we deduce that the resolvent of $\oper{A}^\lambda_\varepsilon$ is compact.
\end{proof}

For each $\varepsilon>0$ the operator $\mathcal{A}^\lambda_\varepsilon$ has a discrete spectrum, and therefore the $n$-approximations of $\mathcal{A}^\lambda_\varepsilon$ and $\mathcal{M}^\lambda_\varepsilon$ may be defined analogously to \eqref{Andef} via the projection operators
\begin{equation*}
\begin{aligned}
\oper{G}^\lambda_{\varepsilon,n}&:\mathfrak{H}\to\mathfrak{H}\\
e_{\varepsilon,k}^\lambda&\mapsto\begin{cases}
e_{\varepsilon,k}^\lambda&\text{ if }k\le n,\\
0&\text{ if }k>n,
\end{cases}
\end{aligned}
\end{equation*}
(where $\{e_{\varepsilon,k}^\lambda\}_{k\in\mathbb{N}}$ are normalised eigenfunctions of $\mathcal{A}^\lambda_\varepsilon$) as
\begin{equation*}\label{Anepsilondef}
\oper{A}^\lambda_{\varepsilon,n}=\oper{G}^\lambda_{\varepsilon,n}\oper{A}^\lambda_\varepsilon\oper{G}^\lambda_{\varepsilon,n}\quad\text{and}\quad\oper{M}^\lambda_{\varepsilon,n}=\oper{G}^\lambda_{\varepsilon,n}\oper{M}^\lambda_\varepsilon\oper{G}^\lambda_{\varepsilon,n}.
\end{equation*}
We know by \autoref{src of truncation} that the family $\{\oper{A}_{\varepsilon,n}^\lambda\}_{\lambda\in[0,1],n\in\overline{\mathbb{N}}}$ is continuous in the strong resolvent sense. In addition, we have:
\begin{lemma}\label{src in epsilon}
The family $\{\oper{A}_\varepsilon^\lambda\}_{\lambda\in[0,1],\varepsilon\in[0,\infty)}$ is continuous in the strong resolvent sense.
\end{lemma}
\begin{proof}
By the equivalence of strong and weak convergence of the resolvent for self-adjoint operators  \cite[VIII, Problem 20(a)]{Reed1981} it is sufficient to prove that $(\oper{A}^\lambda_\varepsilon+1)^{-1}$ is weakly continuous jointly in $\lambda$ and $\varepsilon$. Without loss of generality we restrict to $\varepsilon\in[0,1]$ the general case being no harder. Let $U\subseteq D$ be an open set containing the interval $[0,1]$ such that for $\lambda\in U$, $\operatorname{Re}\sform{a}^\lambda\ge 1$ and $\operatorname{Re}\sform{w}^\lambda\ge-1$. Then, for $\lambda\in U$ and $\varepsilon\in[0,1]$ the forms $\sform{a}^\lambda_\varepsilon$ are closed and sectorial, with $\operatorname{Re}\sform{a}^\lambda_\varepsilon\ge0$. Hence the associated operators have the resolvent bound $\norm{(\oper{A}^\lambda_\varepsilon+\zeta)^{-1}}_{\bounded{\spac{H}}}\le 1/\operatorname{Re}\zeta$ for $\operatorname{Re}\zeta>0$. In particular,
 \begin{equation}\label{eq:resolvent-estimate-for-src-in-epsilon}
 \sup_{\varepsilon\in[0,1],\lambda\in U}\norm{(\oper{A}^\lambda_{\varepsilon}+1)^{-1}}_{\bounded{\spac{H}}}\le1.
 \end{equation}
 Now fix $u,v\in \spac{H}$, let $\varepsilon_n\to\varepsilon_\infty\in[0,\infty)$ and define the sequence of holomorphic functions $f_n:U\to\mathbb{C}$ by
\begin{equation*}
f_n(\lambda)=\ip{(\oper{A}^\lambda_{\varepsilon_n}+1)^{-1}u-(\oper{A}^\lambda_{\varepsilon_\infty}+1)^{-1}u}{v}
\end{equation*}
with $f_\infty=0$. To prove the joint weak continuity of the resolvent it is clearly sufficient to show that $f_n\to 0$ uniformly over $\lambda\in[0,1]$. The case $\varepsilon_\infty>0$ is straightforward so we assume that $\varepsilon_\infty=0$. Without loss of generality we may assume that $\varepsilon_n\ne0$ for all $n$. We will use a simple corollary of Montel's theorem (see e.g. \cite[Theorem 14.6]{Rudin1987}) that states that a sequence of holomorphic functions that is uniformly bounded on an open set $U\subseteq\mathbb{C}$ and converges pointwise in $U$ converges uniformly on any compact set $K\subset U$. The uniform boundedness of $f_n$ follows from \eqref{eq:resolvent-estimate-for-src-in-epsilon} above. Thus it suffices to show that $f_n\to0$ pointwise. To this end we will establish pointwise convergence of the corresponding forms $\sform{A}^\lambda_{\varepsilon_n}$. Indeed,
\begin{equation*}
\forall \lambda\in D, w\in\dom{\sform{a}^\lambda_{\varepsilon_n}},\quad\sform{a}^\lambda_{\varepsilon_n}[w]-\sform{a}^\lambda[w]=\varepsilon_n\sform{w}^\lambda[w]\to 0\text{\quad as }n\to\infty.
\end{equation*}
For $n\in\mathbb{N}$ the forms have common form domain $\dom{\sform{a}}\cap\dom{\sform{w}}$, which is a form core for $\sform{a}^\lambda$, and the sequence of form differences $\sform{a}^\lambda_{\varepsilon_n}-\sform{a}^\lambda$ is uniformly sectorial. Thus  due to \cite[VIII.\S3.2-Theorem 3.6]{Kato1995} $\oper{A}^\lambda_{\varepsilon_n}\tosr \oper{A}^\lambda$ as $n\to\infty$, which implies the pointwise convergence $f_n\to 0$ and completes the proof.
\end{proof}

\begin{corollary}\label{cor:mlambda}
The family $\{\oper{M}_\varepsilon^\lambda\}_{\lambda\in[0,1],\varepsilon\in[0,\infty)}$ is continuous in the strong resolvent sense.
\end{corollary}
\begin{proof}
This follows from the stability of strong resolvent continuity with respect to bounded strongly continuous perturbations.
\end{proof}

\section{Proof of Theorem \ref{thm:main2}}\label{sec:proof}
We split the proof into first proving upper  and lower semi-continuity of $\Sigma(\cdot,\cdot)$ and of $\Sigma_\varepsilon(\cdot,n)$. Informally, we recall that \emph{upper}-semicontinuity of spectra means that the spectrum cannot expand when perturbed, while \emph{lower}-semicontinuity means that the spectrum cannot shrink when perturbed. Then, the uniform convergence in $\lambda\in[0,1]$ of  $\Sigma_\varepsilon(\lambda,n)\to\Sigma(\lambda,\varepsilon)$ as $n\to\infty$ is addressed.

\begin{proof}[Proof of \autoref{thm:main2}.]
\emph{1) Lower semi-continuity.} The lower semi-continuity of spectra under strong resolvent convergence of self-adjoint operators is standard (e.g. \cite[VIII.\S1.2-Theorem 1.14.]{Kato1995}). As  $\{\oper{M}^{\lambda}_{\varepsilon}\}_{\lambda\in[0,1],\varepsilon\in[0,\infty)}$ is continuous in the strong resolvent sense (\autoref{cor:mlambda}) we have that $\Sigma$ is lower semi-continuous.

Now let us consider  $\Sigma_\varepsilon$. For fixed $n$, $\Sigma_\varepsilon(\cdot,n)$ is associated to a finite dimensional operator and hence is clearly lower semi-continuous (and, in fact, continuous). However, let us also consider what happens as $n$ varies. This requires some caution  due to the spurious eigenvalue of $\oper{M}^{\lambda}_{\varepsilon,n}$ at $0$ for $n<\infty$ (see \autoref{rek:thm} for further discussion). We instead consider the operator $\widehat{\oper{M}}^\lambda_{\varepsilon,n}:=\oper{M}^\lambda_{\varepsilon,n}+M(1-\oper{G}^\lambda_{\varepsilon,n}):\mathfrak{H}\to\mathfrak{H}$ where $M>1$ is arbitrary (note that $\widehat{\oper{M}}^\lambda_{\varepsilon,\infty}=\oper{M}^\lambda_{\varepsilon,\infty}$). This moves the spurious eigenvalue to $M\not\in(-\infty,1]$. By \autoref{src of truncation}, along any sequence $\lambda_n\to\lambda\in[0,1]$ as $n\to\infty$ we have $\oper{M}^{\lambda_n}_{\varepsilon,n}\tosr \oper{M}_{\varepsilon}^\lambda$ as $n\to\infty$. Thanks to the stability of strong resolvent convergence with respect to strongly continuous bounded perturbations we also have $\widehat{\oper{M}}^{\lambda_n}_{\varepsilon,n}\tosr \oper{M}^\lambda_\varepsilon$. Moreover, the spectra of  $\widehat{\oper{M}}^\lambda_{\varepsilon,n}$ and $\widetilde{\oper{M}}^\lambda_{\varepsilon,n}$ agree in $(-\infty,1]$ as $M>1$. We have therefore established that given any $\delta>0$ there exists $N>0$ such that for all $n>N$ any point in $\Sigma(\lambda,\varepsilon)$ is within $\delta$ of a point in $\Sigma_\varepsilon(\lambda_n,n)$.

\emph{2) Upper semi-continuity} follows from  \autoref{prop:compactness1} below. Moreover, it follows from  \autoref{prop:compactness1} that given any $\delta>0$ there exists $N>0$ such that for all $n>N$ any point in $\Sigma_\varepsilon(\lambda_n,n)$  is within $\delta$ of a point in $\Sigma(\lambda,\varepsilon)$.

\emph{3)} Note that from (1) and (2) it follows  that $\Sigma_\varepsilon(\lambda_n,n)\to \Sigma(\lambda,\varepsilon)$ for any sequence $\lambda_n$ that converges to $\lambda$ as $n\to\infty$.

\emph{4) The uniform convergence  of $\Sigma_{\varepsilon}(\cdot,n)\to \Sigma(\cdot,\varepsilon)$ as $n\to\infty$} follows from (3) combined with the fact that $[0,1]$ is compact. Indeed, by contradiction,  if uniform convergence didn't hold, then there would exist a $\delta>0$ such that for every $N$ there would exist $n>N$ such that $d_H(\Sigma_{\varepsilon}(\lambda_n,n),\Sigma(\lambda_n,\varepsilon))>\delta$ for some $\lambda_n\in[0,1]$. By compactness there exists a subsequence (we abuse notation and keep the index $n$) along which $\lambda_n\to\lambda_\infty\in[0,1]$. As $\Sigma(\cdot,\varepsilon)$ is continuous, for all sufficiently large $n$ we must have $d_H(\Sigma(\lambda_n,\varepsilon),\Sigma(\lambda_\infty,\varepsilon))<\delta/2$. Therefore, it must also hold that $d_H(\Sigma_{\varepsilon}(\lambda_n,n),\Sigma(\lambda_\infty,\varepsilon))>\delta/2$ for infinitely many $n$'s. However, this is a contradiction to (3).
\end{proof}
The missing ingredient in the above proof is:
\begin{proposition}\label{prop:compactness1}
Let $\sigma_n\to \sigma$ as $n\to\infty$ with $\sigma_n,\sigma\in(-\infty,1]$ and $\lambda_n\to\lambda$ as $n\to\infty$ with $\lambda_n,\lambda\in[0,1]$. Then the following hold.
\begin{enumerate}
\item Let $\varepsilon_n\to\varepsilon\ge0$ as $n\to\infty$, and $\{u_n\}_{n=1}^\infty$ be a sequence with $\norm{u_n}=1$, $u_n\in\dom{\oper{M}^{\lambda}_{\varepsilon_n}}$ and $\oper{M}^{\lambda_n}_{\varepsilon_n}u_n=\sigma_nu_n$. Then $\{u_n\}_{n=1}^\infty$ has a subsequence strongly converging to some $u\ne0$, which satisfies $\oper{M}^\lambda_\varepsilon u=\sigma u$.
\item Let $\varepsilon>0$ be fixed, and $\{u_n\}_{n=1}^\infty$ be a sequence with $\norm{u_n}=1$, $\oper{G}^{\lambda_n}_{\varepsilon,n}u_n=u_n$ and $\oper{M}^{\lambda_n}_{\varepsilon,n}u_n=\sigma_nu_n$. Then  $\{u_n\}_{n=1}^\infty$ has a subsequence strongly converging to some $u\ne0$, which satisfies $\oper{M}^\lambda_\varepsilon u=\sigma u$.
\end{enumerate} 
\end{proposition}
\begin{proof}
As the proof of the first claim is slightly simpler and otherwise the same, we only give the proof for the second claim, leaving the first to the reader. Each $u_n$ solves the equation
\begin{equation*}
\oper{G}^{\lambda_n}_{\varepsilon,n}\oper{A}^{\lambda_n}_{\varepsilon}\oper{G}^{\lambda_n}_{\varepsilon,n}u_n-\sigma_n u_n+\oper{G}^{\lambda_n}_{\varepsilon,n}\oper{K}^{\lambda_n}\oper{G}^{\lambda_n}_{\varepsilon,n}u_n=0.
\end{equation*}
The requirement that $u_n=\oper{G}^{\lambda_n}_{\varepsilon,n}u_n$ and the fact that $\oper{G}^{\lambda_n}_{\varepsilon,n}$ commutes with $\oper{A}^{\lambda_n}_{\varepsilon}$ means that this is equivalent to
\begin{equation}\label{eq:eigenfunction equation for u_k}
{\oper{A}^{\lambda_n}_{\varepsilon}u_n=\sigma_n u_n-\oper{G}^{\lambda_n}_{\varepsilon,n}\oper{K}^{\lambda_n}u_n.}
\end{equation}
Taking the inner product with $u_n$ we estimate,
\begin{equation} \label{eq:form estimate for eigenfunctions1}
\sform{A}^0[u_n]\le C\sform{A}^{\lambda_n}[u_n]\le C\sform{A}^{\lambda_n}_{\varepsilon_n}[u_n]\le C\sigma_n\norm{u_n}^2+C\sup_{\lambda\in[0,1]}\norm{\oper{K}^\lambda}_{\bounded{\mathfrak{H}}}\norm{u_n}^2\le C'
\end{equation}
where $C$ is independent of $n$ {and} comes from the relative form boundedness of the holomorphic family $\{\oper{A}^\lambda\}_{\lambda\in D}$ (see \cite[VII-\S4.2]{Kato1995}) and the supremum is finite by the uniform boundedness principle as $\{\oper{K}^\lambda\}_{\lambda\in[0,1]}$ is strongly continuous. Hence for all $n$ we have $\norm{ |\oper{A}^0|^{1/2}u_n}^2\le C'$, where $|\oper{A}^0|^{1/2}$ is the square root of the positive self-adjoint operator $\oper{A}^0$. By assumption, $\oper{P}$ is relatively compact with respect to $\oper{A}^0$, and hence also to $|\oper{A}^0|^{1/2}$.
{Indeed, the inverse of $|\oper{A}^0|^{1/2}$ can be expressed using the functional calculus (see \cite[V-\S3.11-Equation 3.43]{Kato1995}) of the self-adjoint operator $\oper{A}^0$ as
\begin{equation*}
|\oper{A}^0|^{-1/2}=\frac1\pi\int^\infty_0\zeta^{-1/2}(\oper{A}^0+\zeta)^{-1}\,d\zeta
\end{equation*}
where the integral is absolutely convergent in operator norm due to the bound $\norm{(\oper{A}^0+\zeta)^{-1}}_{\bounded{\spac{H}}}\le (1+\zeta)^{-1}$ for $\zeta\ge0$. By composing both sides of this equation on the left with $\oper{P}$ and moving $\oper{P}$ inside the integral (which is possible as $\oper{P}$ is bounded and the integral converges absolutely in norm) we deduce that $\oper{P}|\oper{A}^0|^{-1/2}$ is given by an absolutely norm convergent integral of compact operators, and is hence compact.}

Thus we may pass to a subsequence (though we retain the subscript $n$) for which
\begin{equation*}
 \oper{P}u_n\to v\in\mathfrak{H}.
\end{equation*}
Then by rewriting \eqref{eq:eigenfunction equation for u_k} and using $\oper{K}^\lambda=\oper{K}^\lambda\oper{P}$ for all $\lambda\in[0,1]$ we have
\begin{equation}\label{eq:rewritten eigenfunction equation for u_k}
 u_n=-(\oper{A}^{\lambda_n}_{\varepsilon}-\sigma_n)^{-1}\oper{G}^{\lambda_n}_{\varepsilon,n}\oper{K}^{\lambda_n}\oper{P}u_n
\end{equation}
where the resolvent exists by the assumption that $\oper{A}^\lambda\ge1+\alpha$ for all $\lambda\in[0,1]$. As remarked before $\oper{G}^{\lambda}_{\varepsilon,n}\tos 1$ uniformly in $\lambda\in[0,1]$ so that $\oper{G}^{\lambda_n}_{\varepsilon,n}\tos1$ as $n\to\infty$. Therefore by the composition of strong convergences
\begin{equation*}
u_n\to -(\oper{A}^{\lambda}_{\varepsilon}-\sigma)^{-1}\oper{K}^\lambda v:=u
\end{equation*}
as $n\to\infty$. Then as $u_n$ is strongly convergent, necessarily $v=\oper{P}u$ and the assertion of the proposition follows.
\end{proof}

\section{Non-positive operators: proof of Theorem \ref{thm:main}}\label{sec:nonpositive}
We define the $\varepsilon$-approximations of $\oper{A}^\lambda_\pm$ as before in terms of a pair of holomorphic families $\oper{W}^\lambda_\pm$ with the same assumptions. The eigenprojections of $\oper{A}^\lambda_{\varepsilon}$ are then denoted by $\oper{G}^\lambda_{\pm,\varepsilon,n}$ and we define
\begin{equation*}
 \oper{G}^\lambda_{\varepsilon,n}=\begin{bmatrix}
                     \oper{G}^\lambda_{+,\varepsilon,n}&0\\
		      0&\oper{G}^\lambda_{-,\varepsilon,n}
                    \end{bmatrix}
\end{equation*}
and
\begin{equation*}
\begin{aligned}
 \oper{A}^\lambda_{\varepsilon,n}&=\oper{G}^\lambda_{\varepsilon,n}\oper{A}^\lambda_{\varepsilon}\oper{G}^\lambda_{\varepsilon,n}\\
 \oper{M}^\lambda_{\varepsilon,n}&=\oper{G}^\lambda_{\varepsilon,n}\oper{M}^\lambda_{\varepsilon}\oper{G}^\lambda_{\varepsilon,n}.
\end{aligned}
\end{equation*}
All the preceding proofs of continuity can be adapted to this case. Indeed, \autoref{lem:cont-alambda} holds without modification, while \autoref{src of truncation} and \autoref{src in epsilon} can be extended by using the identity
\begin{equation*}
\left(\begin{bmatrix}
\oper{T}_+&0\\
0&\oper{T}_- 
\end{bmatrix}
+i\right)^{-1}=
\begin{bmatrix}
(\oper{T}_++i)^{-1}&0\\
0&(\oper{T}_-+i)^{-1}
\end{bmatrix}
\end{equation*}
and the stability of norm (resp. strong) continuity to symmetric bounded norm (reps. strongly) continuous perturbations. With these continuity results, the proof of lower semi-continuity of $\Sigma$ and $\Sigma_\varepsilon$ can be easily adapted. The compactness result \autoref{prop:compactness1} that establishes the upper semi-continuity needs a little more modification. Recall that the discrete region of the spectrum is the gap $(-\alpha-1,1+\alpha)$ rather than the half-line $(-\infty,1+\alpha)$. We restate the compactness result below.
\begin{proposition}\label{prop:compactness12}
Let $\sigma_n\to \sigma$ as $n\to\infty$ with $\sigma_n,\sigma\in[-1,1]$ and $\lambda_n\to\lambda$ as $n\to\infty$ with $\lambda_n,\lambda\in[0,1]$. Then the following hold.
\begin{enumerate}
\item Let $\varepsilon_n\to\varepsilon\ge0$ as $n\to\infty$, and $\{u_n\}_{n=1}^\infty$ be a sequence with $\norm{u_n}=1$, $u_n\in\dom{\oper{M}^{\lambda}_{\varepsilon_n}}$ and $\oper{M}^{\lambda_n}_{\varepsilon_n}u_n=\sigma_nu_n$. Then $\{u_n\}_{n=1}^\infty$ has a subsequence strongly converging to some $u\ne0$, which satisfies $\oper{M}^\lambda_\varepsilon u=\sigma u$.
\item Let $\varepsilon>0$ be fixed, and $\{u_n\}_{n=1}^\infty$ be a sequence with $\norm{u_n}=1$, $\oper{G}^{\lambda_n}_{\varepsilon}u_n=u_n$ and $\oper{M}^{\lambda_n}_{\varepsilon,n}u_n=\sigma_nu_n$. Then  $\{u_n\}_{n=1}^\infty$ has a subsequence strongly converging to some $u\ne0$, which satisfies $\oper{M}^\lambda_\varepsilon u=\sigma u$.
\end{enumerate} 
\end{proposition}
\begin{proof}[Proof (sketched)]
We need only change \eqref{eq:form estimate for eigenfunctions1} to the two estimates
\begin{align*}
\sform{A}^0_\pm[u^\pm_k]&\le C_\pm\sform{A}_\pm^{\lambda_k}[u^\pm_k]\le C_\pm\sform{A}^{\lambda_k}_{\pm,\varepsilon_k}[u^\pm_k]\\
&\le C_\pm|\sigma_k|\norm{u^\pm_k}^2+C_\pm\sup_{\lambda\in[0,1]}\norm{\oper{K}^\lambda}_{\bounded{\mathfrak{H}}}\norm{u_k}^2\le C'
\end{align*}
obtained by taking the inner product of \eqref{eq:eigenfunction equation for u_k} with $u^\pm_k$ where $u_k=(u^+_k,u^-_k)\in\mathfrak{H}_+\times\mathfrak{H}_-$, from which the relative compactness of $\oper{P}u_k$ follows as before, and lastly note that $\oper{A}^\lambda_\pm\ge1+\alpha$ implies that the resolvent $(\oper{A}^{\lambda_k}_{\varepsilon_k}-\sigma_k)^{-1}$ exists in \eqref{eq:rewritten eigenfunction equation for u_k}.
\end{proof} 
This proves \autoref{thm:main}.

\section{An application: plasma instabilities}\label{sec:application}
The discussion in this section is informal. As stability analysis typically relies on a detailed understanding of the spectrum of the linearised problem, most results in this direction require delicate spectral analysis. However, an outstanding open problem has been stability analysis of plasmas that do not possess special symmetries (such as periodicity or monotonicity\footnote{\emph{Monotonicity}, roughly speaking, means that there are fewer particles at higher energies. For a precise definition see e.g. \cite{Ben-Artzi2011b}.}) due to the more complicated structure of the spectrum. A significant obstacle has been the existence of an essential spectrum extending to both $\pm\infty$. Let us briefly outline the problem, which is treated in detail in  \cite{Ben-Artzi2013a}.

Plasmas are typically modelled by the relativistic Vlasov-Maxwell system: Letting $f=f(t,x,v)$ be a probability density function measuring the density of electrons that at time $t\geq0$ are located at the point $x\in\mathbb{R}^d$, have momentum $v\in\mathbb{R}^d$ and velocity $\hat{v}=v/\sqrt{1+|v|^2}$, the (relativistic) Vlasov equation
	\begin{equation}\label{eq:vlasov}
	\frac{\partial f}{\partial t}+\hat{v}\cdot\nabla_x f+\mathbf{F}\cdot\nabla_v f=0
	\end{equation} 
is a transport equation describing their evolution due to the Lorentz force $\mathbf{F}=-\mathbf{E}-\hat{v}\times\mathbf{B}$. Here we have taken the mass of the electrons and the speed of light to be $1$ for simplicity. The fields $\mathbf{E}=\mathbf{E}(t,x)$ and $\mathbf{B}=\mathbf{B}(t,x)$ are the (self-consistent) electric and magnetic fields, respectively. They satisfy Maxwell's equations (written here for their respective potentials $\phi$ and $\mathbf{A}$, satisfying $\mathbf{E}=-\nabla\phi$ and $\mathbf{B}=\nabla\times\mathbf{A}$ in the Lorenz gauge $\partial_t\phi+\nabla\cdot\mathbf{A}=0$):
	\begin{equation}\label{eq:maxwell}
	\left\{\begin{aligned}
	(-\mathbf{\Delta}+\partial_t^2)\mathbf{A}+\mathbf{j}=\mathbf{0},\\
	(\Delta-\partial_t^2)\phi+\rho=0,
	\end{aligned}\right.
	\end{equation}
where $\rho=\rho(t,x)=-\int f\;dv$ is the charge density and ${\mathbf j}={\mathbf j}(t,x)=-\int \hat{v} f\;dv$ is the current density (negative signs are due to the electrons  charge). Linearising \eqref{eq:vlasov} we obtain
	\begin{equation}\label{eq:lin-vlasov}
	\frac{\partial f}{\partial t}+\hat{v}\cdot\nabla_x f+\mathbf{F^0}\cdot\nabla_v f=-\mathbf{F}\cdot\nabla_v f^0,
	\end{equation} 
where $f^0$ and $\mathbf{F^0}$ are the equilibrium density and force field, respectively, and $f$ and $\mathbf{F}$ are their first order perturbations. Maxwell's equations do not require linearisation as they are already linear. We seek solutions to \eqref{eq:maxwell}-\eqref{eq:lin-vlasov} that grow exponentially in time. Therefore, substituting into \eqref{eq:lin-vlasov} the ansatz that all time-dependent quantities behave like $e^{\lambda t}$ with $\lambda>0$, we get
	\begin{equation*}\label{eq:lin-vlasov-ansatz}
	\lambda f+\hat{v}\cdot\nabla_x f+\mathbf{F^0}\cdot\nabla_v f=-\mathbf{F}\cdot\nabla_v f^0.
	\end{equation*} 
An inversion of this equation leaves us with the integral expression
	\begin{equation}\label{eq:f}
	f=-(\lambda+(\hat{v},\mathbf{F^0})\cdot\nabla_{x,v})^{-1}(\mathbf{F}\cdot\nabla_v f^0)
	\end{equation}
which depends upon $\lambda$ as a parameter. By substituting the expression \eqref{eq:f} into Maxwell's equations \eqref{eq:maxwell}, $f$ is eliminated as an unknown, and the only unknowns left are  $\phi$ and $\mathbf{A}$. Note that an immediate benefit is that the problem now only involves the spatial variable $x$, and not the full phase-space variables $x,v$.

We are therefore left with the task of showing that Maxwell's equations are satisfied with the parameter $\lambda>0$.  Gauss' equation, for instance, becomes
	\begin{equation*}
	(\Delta-\lambda^2)\phi=-\rho=\int f\ dv=-\int(\lambda+(\hat{v},\mathbf{F^0})\cdot\nabla_{x,v})^{-1}(\mathbf{F}\cdot\nabla_v f^0)\ dv
	\end{equation*}
which is an equation of the form
	\begin{equation}\label{eq:gauss}
	(\Delta-\lambda^2)\phi+\mathcal{K}^\lambda_{--}\phi+\mathcal{K}^\lambda_{-+}\mathbf{A}=0,
	\end{equation}
{where, for instance, 
	\begin{align*}
	\mathcal{K}^\lambda_{--}\phi&=\int(\lambda+(\hat{v},\mathbf{F^0})\cdot\nabla_{x,v})^{-1}(\nabla\phi\cdot\nabla_v f^0)\ dv,\\
	\mathcal{K}^\lambda_{-+}\mathbf{A}&=\int(\lambda+(\hat{v},\mathbf{F^0})\cdot\nabla_{x,v})^{-1}((\hat{v}\times(\nabla\times\mathbf{A}))\cdot\nabla_v f^0)\ dv.
	\end{align*}}
The rest of Maxwell's equations can be written as
	\begin{equation}\label{eq:maxwell2}
	(-\mathbf{\Delta}+\lambda^2)\mathbf{A}+\mathcal{K}^\lambda_{+-}\phi+\mathcal{K}^\lambda_{++}\mathbf{A}=\mathbf{0}.
	\end{equation}
{(we omit the precise form of these operators here).
The system \eqref{eq:gauss}-\eqref{eq:maxwell2} for $\phi$ and $\mathbf{A}$ turns out to be self-adjoint and  is precisely of the form \eqref{eq:non-positive}. Exhibiting linear instability, i.e. the existence of a growing mode with rate $\lambda>0$, is equivalent to solving this system for  some $\lambda>0$. The operator in this system has the form
	\begin{equation*}
	\oper{M}^\lambda=\oper{A}^\lambda+\oper{K}^\lambda=
	\begin{bmatrix}
	-\mathbf{\Delta}+\lambda^2&0\\
	0&\Delta-\lambda^2
	\end{bmatrix}
	+
	\begin{bmatrix}
	\oper{K}^\lambda_{++}&\oper{K}^\lambda_{+-}\\
	\oper{K}^\lambda_{-+}&\oper{K}^\lambda_{--}
	\end{bmatrix},\quad \lambda>0.
	\end{equation*}

Hence now one would like to show that for some $\lambda>0$, the operator $\oper{M}^\lambda$ has a nontrivial kernel. As this operator is self-adjoint for all $\lambda>0$, its spectrum lies on the real line. We use this fact to ``track'' the spectrum as $\lambda$ varies from $0$ to $+\infty$ and find an eigenvalue that crosses through $0$. By adding to $\oper{A}^\lambda$ the operator
	\begin{equation*}
	\oper{W}=\begin{bmatrix} 1+x^2&0\\0&-1-x^2\end{bmatrix}
	\end{equation*}
and defining
	\begin{equation*}
	\oper{M}^\lambda_\varepsilon=\oper{A}^\lambda+\varepsilon\oper{W}+\oper{K}^\lambda,\quad\lambda>0,\,\varepsilon>0
	\end{equation*}
we obtain a family of operators with a compact resolvent. This family enjoys the properties that we studied in this paper. For instance, natural candidates for the projection operators $\oper{P}_\pm$ are multiplications by the indicator functions (in the appropriate spaces) onto the (compact) support of the steady-state around which we linearise.

Let us describe the method for finding a nontrivial kernel in a nutshell.  It is shown that there exist $0<\lambda_*<\lambda^*<\infty$ (independent of $n$ and $\varepsilon$) for which the corresponding approximate operators $\mathcal{M}^{\lambda_*}_{\varepsilon,n}$ and  $\mathcal{M}^{\lambda^*}_{\varepsilon,n}$ have a different number of negative (and positive) eigenvalues, and therefore due to the continuous dependence of the spectrum (as a set) on the parameter $\lambda$ there must exist $\lambda_*<\lambda_n<\lambda^*$ for which $\mathcal{M}^{\lambda_n}_{\varepsilon,n}$ has a nontrivial kernel. Since $\lambda_n$ is a bounded sequence, one can extract a convergent subsequence converging, say, to some $\lambda_\infty\in[\lambda_*,\lambda^*]$. \autoref{thm:main} is then invoked to show that one can also take the two limits $n\to\infty$ and $\varepsilon\to0$ to conclude that $\oper{M}^{\lambda_\infty}$ has a nontrivial kernel.
We refer to \cite{Ben-Artzi2013a} for full details.}

\bibliography{library}
\bibliographystyle{abbrv}
\end{document}